\documentclass[absolute]{mymathart}
\usepackage{mymathmacros}

 \newtheoremstyle{numberedstyle}
   {9pt}
   {9pt}
   {\normalfont}
   {}
   {\bfseries}
   {.}
   {\newline}
   {}

\newtheorem{thm}{Theorem}[section]%
\newtheorem{prop}[thm]{Proposition}%

\theoremstyle{numberedstyle}

\newcommand{\B}{\mathcal{B}}
\renewcommand{\H}{\mathbb{H}}

\usepackage{hyperref}

\title{The escaping set of the exponential}
\author{Lasse Rempe}
\address{Dept.\ of Mathematical Sciences, University of Liverpool, Liverpool L69 7ZL, UK}
\email{l.rempe@liverpool.ac.uk}
\thanks{Supported by EPSRC Fellowship EP/E052851/1.}

\newcommand{\s}{\underline{s}}

\begin{document} 
 \begin{abstract}
  We show that the set $I(f)$ of points that converge to infinity under
   iteration of the exponential map $f(z)=\exp(z)$ is a connected subset
   of the complex plane. 
 \end{abstract}
 
 \maketitle

 \section{Introduction}
  If $f:\C\to\C$ is an entire transcendental function, then its
   \emph{escaping set} $I(f)$ is the set of points that
   tend to infinity under iteration:
   \[ I(f) = \{z\in\C: f^n(z)\to \infty\}. \]

  For the dynamically simplest entire functions, such as 
   exponential maps of the form $f(z)=\exp(z)+a$ with $a<-1$, 
   the escaping set is the disjoint union of uncountably many
   curves to infinity, each of which is a connected component of $I(f)$.
   (In particular, $I(f)$ is disconnected while $I(f)\cup\{\infty\}$ is
    connected and path-connected.) Eremenko
    \cite{alexescaping} conjectured that every connected component of
    $I(f)$ is unbounded for every transcendental entire function $f$.
    Despite recent progress
    (compare e.g.\ \cite{eremenkoproperty,ripponstallard,strahlen}), 
    this question is still very much open. 

  In view of this, the escaping set is usually viewed very much as 
   a set that is likely to be disconnected. However,
   Rippon and Stallard
   \cite{ripponstallard} proved that the escaping set of an entire function 
   with a
   multiply-connected wandering domain is in fact connected. They have
   since extended this result to much larger classes of entire functions
   \cite{ripponstallardsmallgrowth}.

 These examples are quite different from the exponential maps 
  mentioned above in that they do not belong to the 
  \emph{Eremenko-Lyubich class} 
    \[ \B := \{f:\C\to\C\text{ transcendental, entire}: 
                \sing(f^{-1})\text{ is bounded}\} \] 
   (where $\sing(f^{-1})$ denotes the set of critical and asymptotic 
   values of $f$). We note that, if $f\in\B$, then $I(f)$ is a subset of the
   Julia set $J(f)$ 
   \cite[Theorem 1]{alexmisha}. Bergweiler (personal communication) 
   asked whether the escaping set of 
   a function in $\B$ can be connected, and more precisely whether this 
   might be the case for 
   the function $f(z)=\pi\sin(z)$. While Mihaljevic-Brandt 
   \cite{helenaconjugacy} has given a negative answer to the latter,
   Rippon and Stallard observed that, for the function $f(z)=(\cosh(z))^2$, 
   the escaping set is connected. Indeed, 
   the union of the 
   real axis with all its iterated preimages is path-connected
   (and clearly dense in $I(f)$).

 In contrast to this example, for the exponential map
  $f(z)=\exp(z)$ every path-connected
  component of the escaping set is known
  to be a single curve to $\infty$ that is relatively closed and
  nowhere dense in $I(f)$
  (see Proposition \ref{prop:pathcomponents}). 
  It may seem plausible that these path-connected
  components are also the connected components of $I(f)$, but 
  we show that the situation is rather different.

 \begin{thm}[Escaping set of the exponential] \label{thm:main}
  Let $a\in(-1,\infty)$ and consider the function
   $f(z)=\exp(z)+a$. Then 
   $I(f)$ is a connected subset of the plane. 
 \end{thm}

 The proof is elementary; the main idea is to consider a countable sequence of 
  preimage components of the negative real axis that was studied by
  Devaney \cite{devaneyknaster}
  in his construction of an indecomposable continuum. (See
  Figure \ref{fig:devaney}.)
  Each of these components is an arc tending 
  to infinity in both directions, but we shall show that their union
  is connected. Theorem \ref{thm:main} then follows relatively easily.

 \subsection*{Basic notation.} As usual, we denote the complex plane by
  $\C$, and the Riemann sphere by $\Ch=\C\cup\{\infty\}$. Closures and
  boundaries will be understood to be 
  taken in $\C$, unless explicitly stated otherwise. 

 \subsection*{Acknowledgments.} I would like to thank
  Walter Bergweiler, Mary Rees, Gwyneth Stallard and Phil Rippon for
  interesting discussions.

 \section{The Devaney continuum}
  For the rest of the article, fix $a\in (-1,\infty)$ and set
   $f(z) = \exp(z) + a$. Then $f^n(x)\to\infty$ for all $x\in\R$. 

 Let $\H_+$ and $\H_-$ denote the upper and lower half planes, respectively.
  Let $S_+$ denote the strip at imaginary parts between $0$ and $\pi$;
  similarly $S_-$ is the strip at imaginary parts between
  $0$ and $-\pi$. For $\sigma\in\{+,-\}$, let
       \[ L_{\sigma}:\H_{\sigma}\to S_{\sigma} \]
  be the branch of $f^{-1}$ taking values in $S_{\sigma}$.
  $L_{\sigma}$ is a confomal isomorphism that extends to a homeomorphism
  between $\cl{\H_{\sigma}}\setminus\{a\}$ and $\cl{S_{\sigma}}$; we denote
  this extension also by $L_{\sigma}$. 

 Define $\gamma_0^{\sigma} := (-\infty,a)$, and 
   inductively
     \[ \gamma_{k+1}^{\sigma} := L_{\sigma}(\gamma_k^{\sigma}). \]
   Then each $\gamma_{k}^{\sigma}$, $k\geq 1$, is an injective
    curve tending to infinity in both directions. (Also, $\gamma_k^-$ is
    the reflection of $\gamma_k^+$ in the real axis for all $k$.)

   We define sets $\Gamma^{\sigma}$ and $X^{\sigma}$ by
     \[ \Gamma^{\sigma} := \bigcup_{k\geq 0} \gamma_k^{\sigma}, \quad
         X^{\sigma} := \cl{\Gamma^{\sigma}}. \] 
   See Figure \ref{fig:devaney} for a picture of the set
   $X^+$. We require the following key fact
   \cite[p.\ 631]{devaneyknaster} 
  
  \begin{prop}[Hausdorff limit of $\gamma^{\pm}_k$] \label{prop:hausdorff}
   Let $\sigma\in\{+,-\}$. The set $X^{\sigma}\cup\{\infty\}$ 
    is the Hausdorff limit (on the Riemann sphere $\Ch$) 
    of the sequence $(\gamma^{\sigma}_k\cup\{\infty\})$. 

   (In particular, $\bigcup_{k\geq k_0} \gamma^{\sigma}_k$ is dense in
    $X^{\sigma}$ for all $k_0$.)
  \end{prop}
  \begin{proof}
   Let $z_0\in X^{\sigma}\cup\infty$, and let 
    $U$ be a neighborhood of $z_0$ in $\Ch$. We need 
    to show that $\gamma^k\cup U\neq\emptyset$ for all sufficiently large $k$. 

   By definition of $X^{\sigma}$, the set $U$ contains some 
    $z_1\in \Gamma^{\sigma}$. 
    Let $D_n$ denote the (Euclidean) disk of radius
    $2\pi$ around $f^n(z_1)$. It is elementary to see --- using 
    the fact that $f$ is expanding in a suitable right half plane ---
    that there is
    $n_0$ with
    $f^j(a)\notin D_n$ for all $n\geq n_0$ and $j<n$. We may assume that
    $n_0$ is chosen sufficiently large to ensure that also 
    $f^{n}(z_0)\in\R$ for $n\geq n_0$. 

   Hence for $n\geq n_0$, there is a branch
    $\phi_n:D_n\to\C$ of $f^{-n}$ with
    $\phi_n(f^n(z_0))=z_0$. Clearly
    $\max_{z\in D_n}|\phi_n'(z)|\to 0$ as $n\to\infty$ (again due to 
    the expansion of $f$ in a right half plane). In particular,
    there is $k_0\geq n_0+1$ such that for $k\geq k_0$, the image
    of $\phi_{k-1}$ is contained in $U$. We then have 
    $\phi_{k-1}(f^{k-1}(z_0)+\pi i)\in \gamma^{k}\cap U$, 
    and the claim follows. 
  \end{proof} 

 \begin{figure}
  \begin{center}
   \includegraphics[width=.98\textwidth]{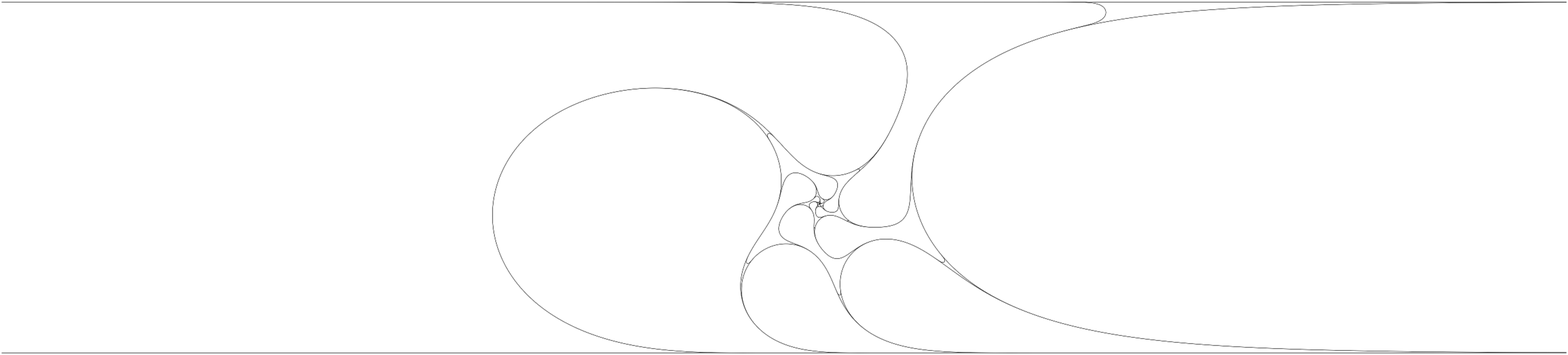}
  \end{center}
  \caption{\label{fig:devaney}The set $X^+$}
 \end{figure}

 \begin{prop}[$\Gamma^{\pm}$ connected] \label{prop:Gamma}
  The sets $\Gamma^+$ and $\Gamma^-$ are connected.
 \end{prop}
 \begin{proof}
  Let $\sigma\in \{+,-\}$.
   Suppose that $U\subset\C$ is an open set with
    $\Gamma^{\sigma}\cap U\neq\emptyset$ and
   $\Gamma^{\sigma}\cap \partial U=\emptyset$. We need to show that 
   $\Gamma^{\sigma}\subset U$.

   Let $z_0\in U\cap\Gamma^{\sigma}$. Then 
   by Proposition \ref{prop:hausdorff}, there is $k_0$ such that
   $\gamma_k\cap U\neq\emptyset$ for all $k\geq k_0$. Since
   $\gamma_k$ is connected, in fact $\gamma_k\subset U$. Thus 
         \[ \Gamma^{\sigma}\subset X^{\sigma} =
             \cl{\bigcup_{k\geq k_0} \gamma^{\sigma}_k} \subset \cl{U}. \] 
    By choice of $U$, we hence have $\Gamma^{\sigma}\subset U$, as desired.
 \end{proof}

 \section{Proof of the Theorem}

  By Proposition \ref{prop:Gamma}, the
    union of $\Gamma^+$ and $\Gamma^-$ connects the horizontal line 
    $\gamma_1^-$ at imaginary
    part $-\pi$ with $\gamma_1^+$ at imaginary part $\pi$. Since
    the set $I(f)$ is $2\pi i$-periodic, it follows that the set 
  \[ Y := \bigcup_{\sigma\in \{+,-\}, k\in\Z} \bigl(\Gamma^{\sigma}+2\pi i k\bigr) \]
  is a connected subset of $I(f)$.
  $Y$ contains all points whose imaginary parts are odd multiples of $\pi$;
  i.e. $f^{-1}\bigl((-\infty,a)\bigr)$. 
  
 \begin{prop}[Preimages of $Y$]
  Set $Y_0 := Y$ and inductively
   $Y_{j+1} := Y_j \cup f^{-1}(Y_j)$. Then $Y_j$ is connected for all $j$.
 \end{prop}
 \begin{proof}
  The proof is by induction on $j$. Note that $Y_j$ contains
  $Y$ for all $j$. 

  Let $k\in\Z$, and let
   $L_k:\C\setminus[a,\infty)\to\C$ be the branch of $f^{-1}$ that
   takes values with imaginary parts between $2\pi k$ and $2\pi(k+1)$. 
   Set $z_k := (2k+1)\pi i$. Then $z_k,f(z_k)\in Y\subset Y_j$,
   and hence $z_k\in Y_j\cap L_k(Y_j)$. As $L_k$ is a continuous function
   (and $Y_j$ is contained in its domain of definition),
   it follows from the induction hypothesis that $Y_j\cup L_k(Y_j)$ is 
   connected. Hence
     \[ Y_{j+1} = \bigcup_{k\in\Z} L_k(Y_j) \cup Y_j \]
   is connected, as claimed. 
 \end{proof}

 \begin{proof}[Proof of Theorem \ref{thm:main}]
  The set $\bigcup_{j\geq 0} f^{-n}(-1)\subset \bigcup_{j\geq 0} Y_j$
   is dense in the Julia set, and hence in the escaping set. Since
   $\bigcup_{j\geq 0} Y_j$ is a connected subset of $I(f)$, the claim follows. 
 \end{proof}

 We contrast our result with the following fact, mentioned in the 
  introduction.

 \begin{prop}[Path-connected components of $I(f)$] \label{prop:pathcomponents}
  Let $P$ be a path-connected component of $I(f)$, where
   again $f(z)=\exp(z)+a$, $a\in(-1,\infty)$. Then $P$ is
   relatively closed and nowhere dense in $I(f)$.
 \end{prop}
 \begin{proof}
  The path-connected components of $I(f)$ (in fact, of the escaping set of
   any exponential map) are completely described in 
   \cite[Corollary 4.3]{markuslassedierk}. First of all, for any $n\geq 0$,
   each connected component
   of $f^{-n}(\R)$ is a path-connected component of $I(f)$.

  Every one of these is nowhere dense and closed in $\C$, and in 
   particular relatively closed in $I(f)$.  

  Now suppose that $z_0\in I(f)$ never maps to the positive real axis
   under iteration. Let $\s=s_0 s_1 s_2 \dots$ be the sequence of
   integers such that $\im f^n(z_0)\in \bigl((2s_n-1)\pi,(2s_n+1)\pi\bigr)$
   for all $n$. It follows from the assumption on $z_0$ that 
   $\s$ must contain infinitely many nonzero entries 
    (see
   \cite[Theorem on p.\ 632]{devaneyknaster}).

  Let $K=K_{\s}$ be the set of all points $z\in J(f)$ with
   $\im f^n(z)\in \bigl[(2s_n-1)\pi,(2s_n+1)\pi\bigr]$ for all $n\geq 0$.
   Clearly $K$ is closed and nowhere dense. It is known 
   \cite{devaneykrych,expescaping} that 
   $K\cap I(f)$ is path-connected; in fact, $K\cap I(f)$ is the 
   trace of an injective curve $g_{\s}:[0,\infty)\to\C$ or 
   $g_{\s}:(0,\infty)\to\C$ with $g_{\s}(t)\to\infty$ as $t\to\infty$. 
   (This curve is called a \emph{Devaney hair} or a \emph{dynamic ray}.) 
   We remark that, for certain addresses $\s$, the 
    limit set of $g_{\s}(t)$ as $t\to 0$ will not consist of a single
    point
    \cite{indecomposable} (compare also \cite{nonlanding}). Any escaping
    points in this limit set must necessarily lie on $g_{\s}$ themselves. 

   By \cite[Corollary 4.3]{markuslassedierk}, the curve
    $P:=K\cap I(f)=g_{\s}$ is the path-connected 
    component of $I(f)$ containing $z_0$. Since $K$ is closed and nowhere
    dense, the claim follows.
 \end{proof}

\bibliographystyle{hamsplain}
\bibliography{/Latex/Biblio/biblio}

\providecommand{\href}[2]{#2}\def\polhk#1{\setbox0=\hbox{#1}{\ooalign{\hidewid%
th \lower1.5ex\hbox{`}\hidewidth\crcr\unhbox0}}}
  \def\polhk#1{\setbox0=\hbox{#1}{\ooalign{\hidewidth\lower1.5ex\hbox{`}\hidew%
idth\crcr\unhbox0}}} \input{cyracc.def} \def\j{{\u i}} \def\J{{\u I}}
  \newfont{\cyrit}{wncyi10 at 12pt}\def\cprime{$'$}
\providecommand{\bysame}{\leavevmode\hbox to3em{\hrulefill}\thinspace}
\begin{thebibliography}{M-B}

\bibitem[D]{devaneyknaster}
Robert~L. Devaney, \emph{Knaster-like continua and complex dynamics}, Ergodic
  Theory Dynam. Systems \textbf{13} (1993), no.~4, 627--634.

\bibitem[DJ]{indecomposable}
Robert~L. Devaney and Xavier Jarque,
  \emph{\href{http://math.bu.edu/people/bob/papers/inde.ps}{Indecomposable
  continua in exponential dynamics}}, Conform. Geom. Dyn. \textbf{6} (2002),
  1--12.

\bibitem[DK]{devaneykrych}
Robert~L. Devaney and Micha{\l} Krych, \emph{Dynamics of {${\rm exp}(z)$}},
  Ergodic Theory Dynam. Systems \textbf{4} (1984), no.~1, 35--52.

\bibitem[E]{alexescaping}
Alexandre~{\`E}. Eremenko, \emph{On the iteration of entire functions},
  Dynamical systems and ergodic theory (Warsaw, 1986), Banach Center Publ.,
  vol.~23, PWN, Warsaw, 1989, pp.~339--345.

\bibitem[EL]{alexmisha}
Alexandre~{\`E}. Eremenko and Mikhail~Yu. Lyubich, \emph{Dynamical properties
  of some classes of entire functions}, Ann. Inst. Fourier (Grenoble)
  \textbf{42} (1992), no.~4, 989--1020.

\bibitem[FRS]{markuslassedierk}
Markus F\"orster, Lasse Rempe, and Dierk Schleicher, \emph{Classification of
  escaping exponential maps}, Preprint, 2004,
  \mbox{\href{http://www.arXiv.org/abs/math.DS/0311427}{arXiv:math.DS/0311427}%
}, to appear in Proc. Amer. Math. Soc.

\bibitem[M-B]{helenaconjugacy}
Helena Mihaljevi\'c-Brandt, \emph{Orbifolds of subhyperbolic transcendental
  maps}, Manuscript, 2008.

\bibitem[R1]{eremenkoproperty}
Lasse Rempe, \emph{On a question of {E}remenko concerning escaping sets of
  entire functions}, Bull. London Math. Soc. \textbf{39} (2007), no.~4,
  661--666,
  \mbox{\href{http://arxiv.org/abs/math.DS/0610453}{arXiv:math.DS/0610453}}.

\bibitem[R2]{nonlanding}
\bysame, \emph{On nonlanding dynamic rays of exponential maps}, Ann. Acad. Sci.
  Fenn. Math. \textbf{32} (2007), 353--369,
  \mbox{\href{http://www.arXiv.org/abs/math.DS/0511588}{arXiv:math.DS/0511588}%
}.

\bibitem[RS1]{ripponstallard}
Philip~J. Rippon and Gwyneth~M. Stallard, \emph{On questions of {F}atou and {E}remenko}, Proc. Amer. Math. Soc.
  \textbf{133} (2005), no.~4, 1119--1126.

\bibitem[RS2]{ripponstallardsmallgrowth}
\bysame, \emph{Escaping points of entire
  functions of small growth}, Math. Z. (to appear),
  \mbox{\href{http://www.arXiv.org/abs/0801.3605}{arXiv:0801.3605}}.


\bibitem[R$^3$S]{strahlen}
G\"unter Rottenfu{\ss}er, Johannes R\"uckert, Lasse Rempe, and Dierk
  Schleicher, \emph{Dynamic rays of entire functions}, Preprint \#2007/05,
  Institute for Mathematical Sciences, SUNY Stony Brook, 2007,
  \mbox{\href{http://www.arXiv.org/abs/0704.3213}{arXiv:0704.3213}}, submitted
  for publication.

\bibitem[SZ]{expescaping}
Dierk Schleicher and Johannes Zimmer, \emph{Escaping points of exponential
  maps}, J. London Math. Soc. (2) \textbf{67} (2003), no.~2, 380--400.

\end{thebibliography}

\end{document}